\documentclass[12pt]{amsart}
\usepackage{amsmath,amsthm,amsfonts,amssymb}

\newcommand{\bbC}{\mathbb{C}}

\newcommand{\bbR}{\mathbb{R}}

\newcommand{\bbF}{\mathbb{F}}

\newcommand{\g}{\mathfrak{g}}
\newcommand{\h}{\mathfrak{h}}

\newcommand{\so}{\mathfrak{so}}
\renewcommand{\sl}{\mathfrak{sl}}
\renewcommand{\sp}{\mathfrak{sp}}
\newcommand{\su}{\mathfrak{su}}

\newcommand{\Stab}{\operatorname{Stab}}

\newcommand{\epi}{\operatorname{epi}}

\newcommand{\rank}{\operatorname{rank}}

\newcommand{\Hom}{\operatorname{Hom}}
\newcommand{\Aut}{\operatorname{Aut}}
\newcommand{\Ad}{\operatorname{Ad}}
\newcommand{\compose}{\circ}

\newcommand{\AG}{\mathsf{A}}

\renewcommand{\O}{\mathrm{O}}
\newcommand{\SO}{\mathrm{SO}}

\newcommand{\SU}{\mathrm{SU}}

\newcommand{\GL}{\mathrm{GL}}
\newcommand{\PGL}{\mathrm{PGL}}
\newcommand{\SL}{\mathrm{SL}}

\newcommand{\Sp}{\mathrm{Sp}}

\numberwithin{equation}{section}

\newtheorem{thm}{Theorem}[section]

\newtheorem{cor}[thm]{Corollary}

\newtheorem{prop}[thm]{Proposition}

\newtheorem{lemma}[thm]{Lemma}

\newtheorem{defn}[thm]{Definition}

\newtheorem{lem}[thm]{Lemma}

\newtheorem*{claim*}{Claim}
\begin{document}

\title[Representations of Fuchsian groups]{{Representation Varieties\\ of Fuchsian Groups} \vskip .2in
\centerline{\hbox{\tiny \mdseries \it Dedicated to the memory of Leon Ehrenpreis}}}

\author{Michael Larsen}
\address{Michael Larsen,
Department of Mathematics,
Indiana University,
Bloomington, IN
U.S.A. 47405}

\author{Alexander Lubotzky
}
\address{Alexander Lubotzky, Einstein Institute of Mathematics,
Edmond J. Safra Campus, Givat Ram,
The Hebrew University of Jerusalem,
Jerusalem, 91904, Israel \\
}

\thanks{ML was partially supported by the National Science Foundation and the United States-Israel Binational  Science Foundation.
AL was partially supported by the European Research Council and the Israel Science Foundation.}


\begin{abstract}
We estimate the dimension of varieties of the form
$\Hom(\Gamma,G)$ where $\Gamma$ is a Fuchsian group and $G$ is a simple real algebraic group,
answering along the way a question of I.\ Dolgachev.
\end{abstract}

\maketitle

\section{Introduction}

Let $G$ be an almost simple real algebraic group, i.e., a non-abelian linear algebraic group over $\bbR$
with no proper normal $\bbR$-subgroups of positive dimension.
Let $\Gamma$ be a finitely generated group.  The set of representations
$\Hom(\Gamma, G(\bbR))$ coincides with the set of real points of the
representation variety $X_{\Gamma,G} := \Hom(\Gamma, G)$.
(We note here, that by a \emph{variety}, we mean an affine scheme of finite type over $\bbR$;
in particular, we do not assume that it is irreducible or reduced.)

Let $X^{\epi}_{\Gamma,G}$ denote the Zariski-closure in $X_{\Gamma,G}$ of
the set of \emph{Zariski-dense homomorphisms}
 $\Gamma\to G(\bbR)$, i.e., homomorphisms with Zariski-dense image.
In this paper, we estimate the dimension of $X^{\epi}_{\Gamma,G}$
when $\Gamma$ is a cocompact Fuchsian group.  Our main results
assert that in most cases, this dimension is roughly $(1-\chi(\Gamma))\dim G$,
where $\chi(\Gamma)$ is the Euler characteristic of $\Gamma$.

To formulate our results more precisely, we need some notation and definitions.
A cocompact oriented Fuchsian group $\Gamma$
(and all Fuchsian groups in this paper will be assumed to be cocompact and oriented without further mention)
always admits a presentation of the following kind:
Consider non-negative integers $m$ and $g$ and integers $d_1,\ldots,d_m$
greater than or equal to $2$, such that
\begin{equation}
\label{Euler}
2-2g-\sum_{i=1}^m (1-d_i^{-1})
\end{equation}
is negative.
For some choice of $m$, $g$, and $d_i$, $\Gamma$ has a presentation
\begin{equation}
\label{presentation}
\begin{split}
\Gamma := \langle x_1,\ldots,x_m,y_1,\ldots,y_g,z_1,\ldots&,z_g\mid x_1^{d_1},\ldots,x_m^{d_m},\\
&x_1\cdots x_m [y_1,z_1]\cdots[y_g,z_g]\rangle,
\end{split}
\end{equation}
and its Euler characteristic $\chi(\Gamma)$ is given by (\ref{Euler}).
If $g=0$ in the presentation (\ref{presentation}), we sometimes denote $\Gamma$ by
$\Gamma_{d_1,\ldots,d_m}$.  If, in addition, $m=3$, $\Gamma$ is called a \emph{triangle group}, and its isomorphism class
does not depend on the order of the subscripts.
Note that the parameter $g$ and the multiset $\{d_1,\ldots,d_m\}$ are determined by the isomorphism class of $\Gamma$.
Every non-trivial element of $\Gamma$ of finite order is conjugate to a power of one of the $x_i$, which is an element of order exactly $d_i$.

\begin{defn}
Let $H$ be an almost simple algebraic group.
We say that a Fuchsian group $\Gamma$ is \emph{$H$-dense} if and only if there exists a homomorphism $\phi\colon \Gamma\to H(\bbR)$ such that $\phi(\Gamma)$ is Zariski-dense in $H$ and $\phi$ is injective on all finite cyclic subgroups of $\Gamma$ (equivalently,
$\phi(x_i)$ has order $d_i$ for all $i$).
\end{defn}

We can now state our main theorems.

\begin{thm}
\label{rep-thm}
For every Fuchsian group $\Gamma$ and every integer $n\ge 2$
$$\dim X^{\epi}_{\Gamma,\SU(n)} = (1-\chi(\Gamma)) \dim \SU(n) + O(1),$$
where the implicit constants depend only on $\Gamma$.
\end{thm}

In particular, this answers a question of Igor Dolgachev, proving the existence in sufficiently high degree,
of uncountably many absolutely irreducible, pairwise non-conjugate, representations.

\begin{thm}
\label{sl2-thm}
For every Fuchsian group $\Gamma$  and every split simple real algebraic group $G$,
$$
\dim X^{\epi}_{\Gamma,G} = (1-\chi(\Gamma))\dim G + O(\rank G),
$$
where the implicit constants depend only on $\Gamma$.
\end{thm}

\begin{thm}
\label{so3-thm}
For every $\SO(3)$-dense Fuchsian group $\Gamma$ and every compact simple real algebraic group $G$,
$$\dim X^{\epi}_{\Gamma,G} = (1-\chi(\Gamma))\dim G + O(\rank G),$$
where the implicit constants depend only on $\Gamma$.
\end{thm}

Let us mention here that all but finitely many Fuchsian groups are $\SO(3)$-dense (see  Proposition \ref{base-case} for the complete list of exceptions).

In the special case that $\Gamma$ is a surface group and $\dim G$ is large compared to the genus, a recent result of Kim and Pansu \cite{KP} implies the above dimension estimates without error terms.  In the generality in which we work, it is certainly necessary to have some error term (since $\chi(\Gamma)$ may not be integral), but ours may not be the best possible.

The proofs of our theorems are based on deformation theory.  It is a well-known result of Weil
\cite{Weil} that the Zariski tangent space to $X_{\Gamma, G}$ at any point
$\rho\in  X_{\Gamma,G}(\bbR)$ is equal to the space of 1-cocycles
$Z^1(\Gamma,\Ad\compose\rho)$, where $\Ad\compose \rho$ is the representation of $\Gamma$ on
the Lie algebra $\g$ of $G$ determined by $\rho$.  (For brevity,
we often denote $\Ad\compose \rho$ by
$\g$, where the action of $\Gamma$ is understood.)
In general, the dimension of the tangent space to $X_{\Gamma,G}$ at $\rho$
can be strictly larger than the dimension of a component of $X_{\Gamma,G}$ containing $\rho$,
thanks to obstructions in $H^2(\Gamma,\Ad\compose \rho)$.
Weil showed that if the coadjoint representation $(\Ad\compose \rho)^*$ has no $\Gamma$-invariant vectors,
then $\rho$ is a non-singular point of $X_{\Gamma,G}$, i.e., it lies on a unique component of $X_{\Gamma,G}$
whose dimension is given by
$\dim Z^1(\Gamma,\Ad\compose \rho)$, the dimension of the Zariski-tangent space to $X_{\Gamma,G}$ at $\rho$.
Computing this dimension is easy; the difficulty
is to find $\rho$ for which the obstruction space vanishes.  A basic technique is to
find a subgroup $H$ of $G$ for which the homomorphisms $\Gamma\to H$ are better understood
and to choose $\rho$ to factor through $H$.  In this paper, we make particular use of the
homomorphisms from $H = \AG_n$ to $G = \SO(n-1)$ and of the
principal homomorphisms from $H = \PGL(2)$ and $H = \SO(3)$ to various groups $G$---see \S3 and \S4 respectively.

It is interesting to compare our results (Theorems \ref{rep-thm}--\ref{so3-thm})
to the results of Liebeck and Shalev \cite{LS2}.  They also estimate
$\dim  X_{\Gamma, G}$ (and implicitly $\dim X_{\Gamma,G}^{\epi}$), but their methods work only for genus $g\ge 2$,
while the most difficult (and interesting) case is $g=0$.  On the other hand, their methods work in arbitrary characteristic, while the methods of this paper appear to break down when the characteristic of the field divides the order of some generator $x_i$.
A striking difference is that they deduce their information about $X_{\Gamma, G}$
from deep results on the finite quotients  of $\Gamma$, while we work directly with $X_{\Gamma, G}^{\epi}$
and can deduce that various families of finite groups of Lie type can be realized as quotients of $\Gamma$ (see \cite{LLM}).

It may also be worth comparing our results to those of Benyash-Krivatz, Chernousov, and Rapinchuk
\cite{BCR}, who consider $X_{\Gamma,\SL_n}$ where $\Gamma$ is a surface group.  They not only
compute the dimension but prove a strong rationality result.  It would be interesting to know if similar rationality results hold for more general semisimple groups $G$. \\

The paper is organised as follows. In \S \ref{s:two}, we give a uniform proof of the
upper bound in Theorems \ref{rep-thm}, \ref{sl2-thm} and \ref{so3-thm}. This requires estimating the dimensions of suitable cohomology groups and boils down to finding
lower bounds on dimensions of centralizers.

To prove the lower bounds of these three theorems, we  present in each case a representation of $\Gamma$
which is ``good'' in the sense that it is a non-singular point of the representation variety to which it belongs.
We then compute the dimension of the tangent  space at the good point.
In \S \ref{s:three}, we explain how one can go from a good representation of $\Gamma$ into a smaller group $H$ to a good representation into a larger group $G$.
The initial step of this kind of induction is via a representation of $\Gamma$ into an alternating group, $\SO(3)$, or $\PGL_2(\bbR)$.
We discuss the alternating group strategy in \S3, where we prove Theorem \ref{rep-thm} and begin the proof of Theorem~\ref{sl2-thm}.
In \S4, we discuss the principal homomorphism strategy, treating the remaining cases of Theorem~\ref{sl2-thm}, proving Theorem~\ref{so3-thm}, and proving the existence of dense homomorphisms
from $\SO(3)$-dense Fuchsian groups to  exceptional compact Lie groups (Proposition~\ref{Exceptional}).

Proposition \ref{base-case} in \S5 shows that there are only
six Fuchsian groups which are not $\SO(3)$-dense.   We do not have a good strategy for finding dense homomorphisms from
these groups to compact simple Lie groups, since the methods of \S3 are not effective.  Y. William Yu found explicit surjective homomorphisms, described in the Appendix, from these groups
to small alternating groups, which may serve as base cases for inductively constructing  dense homomorphisms $\Gamma\to G(\bbR)$ for these groups.
We are grateful to him for his help.

All \emph{Fuchsian groups} in this paper are assumed to be cocompact and oriented.
A \emph{variety} is an affine scheme of finite type over $\bbR$.
Its \emph{dimension} is understood to mean its Krull dimension.
\emph{Points} are $\bbR$-points, and \emph{non-singular} points should be understood scheme-theoretically;
i.e., a point $x$ is non-singular if and only if it lies in only one irreducible component $X$,
and the dimension of $X$ equals the dimension of the Zariski-tangent space at $x$.
An \emph{algebraic group} will mean a linear algebraic group over $\bbR$.
Unless otherwise stated, all topological notions will be understood in the sense of
the Zariski-topology.  In particular, a \emph{closed subgroup} is taken to be Zariski-closed.
Note, however, that an algebraic group $G$ is  \emph{compact} if $G(\bbR)$ is so in the real topology.

We would like to thank the referee for a quick and thorough reading of our paper and a number of very helpful comments.

This paper is dedicated to the memory of Leon Ehrenpreis who was a leading figure in Fuchsian groups and was an inspiration in several other directions---not only mathematically.

\section{Upper Bounds}\label{s:two}

We recall some results from \cite{Weil}.  For every finitely generated group $\Gamma$, the Zariski tangent space to $\rho \in X_{\Gamma, G}(\bbR)$ is equal to $Z^1(\Gamma, {\rm Ad}\circ \rho)$ where ${\rm Ad}:  G \rightarrow {\rm Aut}(\mathfrak{g})$ is the adjoint representation of $ G$ on its Lie algebra.  We will often write this more briefly as $Z^1(\Gamma,\g)$.
Note that $\dim Z^1(\Gamma,\g)$ is always at least as great as the dimension of any component of $X_{\Gamma,G}$ in which $\rho$ lies.
Moreover, if $\Gamma$ is a Fuchsian group and  the coadjoint representation $\g^* = ({\rm Ad}\circ \rho)^* $ has no $\Gamma$-invariant vectors, then $\rho$ is a non-singular point of $X_{\Gamma,G}$.

If $V$ denotes any finite dimensional  real vector space $V$ on which $\Gamma$ acts,
then
\begin{equation}
\label{e-dim}
\begin{split}
\dim Z^1(\Gamma&,V) := (2g-1)\dim V + \dim (V^*)^\Gamma + \sum_{j=1}^m (\dim V  - \dim V^{\langle x_j\rangle}).\\
          &= (1-\chi(\Gamma))\dim V + \dim (V^*)^\Gamma + \sum_{j=1}^m \Bigl(\frac{\dim V}{d_j} - \dim V^{\langle x_j\rangle}\Bigr).
\end{split}
\end{equation}

The following proposition essentially gives the upper bounds in Theorems \ref{rep-thm}, \ref{sl2-thm} and  \ref{so3-thm}, since for every irreducible component $C$ of $X_{\Gamma,G}^{\epi}$
there exists a representation
$\rho\colon \Gamma\to G(\bbR)$ with Zariski-dense image in $C(\bbR)$; $\dim Z^1(\Gamma,\g)$ is at least as great as the dimension of any irreducible component of $X_{\Gamma,G}$
to which $\rho$ belongs and therefore at least as great as $\dim C$.

\begin{prop}\label{p:uph1}
For every Fuchsian group $\Gamma$, every reductive $\mathbb{R}$-algebraic group $ G$  with a Lie algebra $\mathfrak{g}$ and every representation $\rho: \Gamma \rightarrow G(\mathbb{R})$ with Zariski dense image, we have:
$$ \dim Z^1(\Gamma,\g)\leq (1-\chi(\Gamma)) \dim  G+(2g+m+\rank G)+\frac32 m \rank G,$$
where $g$ and $m$ are as in (\ref{presentation}).
\end{prop}

\begin{proof}
By Weil's formula (\ref{e-dim}),
\begin{equation}
\label{dim-z1}
\dim Z^1(\Gamma,\g)=(1-\chi(\Gamma))\dim  G+\dim (\mathfrak{g}^*)^\Gamma+\sum_{j=1}^m\left(\frac{\dim  G}{d_j} - \dim \mathfrak{g}^{\langle x_j \rangle}\right).
\end{equation}
Note that if $\mathfrak{g}$ is the real Lie algebra of $G$ then $\mathfrak{g} \otimes_{\mathbb{R}}\mathbb{C}$ is the complex Lie algebra of $ G$. By abuse of notation we will also denote it by $\mathfrak{g}$. Of course they have the same dimensions over $\mathbb{R}$ and $\mathbb{C}$, respectively.

We have the following dimension estimates.

\begin{lem}\label{l:dualg}
Under the above assumptions:
$$\dim (\mathfrak{g}^*)^\Gamma\leq 2g+m+\rank G.$$
\end{lem}

Let us say that an automorphism $\alpha$ of $ G$ of order $k$ is a \textit{pure outer automorphism} of $ G$ if $\alpha^l$ is not inner for any $l$ satisfying $1 \leq l <  k$.

For inner or pure automorphisms we have:

\begin{lem}
\label{l:inoutaut}\label{error}
Let $\alpha$ be either an inner or a pure outer automorphism of $ G$ of order $k$.  Then
\begin{equation}\label{e:error}
\dim {\rm Fix}_{ G}(\alpha) \geq \frac{\dim  G}k-\rank G.\end{equation}
\end{lem}

\begin{lem}\label{l:inoutautgen}
If $ G$ is a complex reductive group and $\alpha$ any automorphism of $ G$ of order $k$, then
$$\dim {\rm Fix}_{ G}(\alpha) \geq \frac{\dim  G}{k}-\frac{3}2 \rank G,$$
where ${\rm Fix}_{ G}(\alpha)$ denotes the subgroup of the fixed points of $\alpha$.
\end{lem}

Plugging the results of Lemmas \ref{l:dualg} and \ref{l:inoutautgen} into (\ref{e-dim}), and noting that $\dim \mathfrak{g}^{\langle x_j\rangle}$ is equal to $\dim {\rm Fix}_{ G}(x_j)$,  we have:
$$ \dim Z^1(\Gamma,\g)\leq (1-\chi(\Gamma)) \dim  G+(2g+m+\rank G)+\frac32 m \rank G.$$
\end{proof}

\begin{proof}[Proof of Lemma \ref{l:dualg}]
The dimension of the $\Gamma$-invariants on $\mathfrak{g}^*$, $\dim (\mathfrak{g}^*)^\Gamma$, is equal to the dimension of the $\Gamma$-coinvariants on $\mathfrak{g}$. As $\Gamma$ is Zariski dense in $ G$, this is equal to the dimension of the coinvariants of $G$ acting on $\mathfrak{g}$ via ${\rm Ad}$. Letting $G^0$ act first, we deduce that the space of $G$-coinvariants is a quotient space of $\mathfrak{g}/[\mathfrak{g},\mathfrak{g}]$. More precisely, it is equal to the coinvariants of  $\mathfrak{g}/[\mathfrak{g},\mathfrak{g}]$ acted upon by the finite group $G/ G^0$. As $\mathfrak{g}/[\mathfrak{g},\mathfrak{g}]$ is a characteristic zero  vector space, the dimension of the coinvariants is the same as  that of  the $G/G^0$-invariant subspace. 	Now, the space of linear maps ${\rm Hom}(\mathfrak{g}/\mathfrak[\mathfrak{g},\mathfrak{g}],\mathbb{R})$ corresponds to the homomorphisms from $G^0$ to $\mathbb{R}$ and the $G/G^0$-invariants are those which can be extended to $G$. So, altogether $\dim (\mathfrak{g}^*)^\Gamma$ is bounded by $\dim {\rm Hom}(G,\mathbb{R})$. Now
$$\dim {\rm Hom}(G,\mathbb{R})=\dim  G^{\mathrm{ab}},$$
where $G^{\mathrm{ab}}=G/ [G,G]$, and
$$G^{\mathrm{ab}}=U\times T \times A,$$
where $U$ is a unipotent group, $T$ a torus, and $A$ a finite group. So $\dim  G^{\mathrm{ab}}=\dim U+\dim T$. As $\Gamma$ is Zariski dense in $ G$, its image is Zariski dense in $U$ and hence
$$\dim U\leq d(\Gamma)\leq 2g+m,$$
where $d(\Gamma)$ denotes the number of generators of $\Gamma$. Now, $T$, being a quotient of $G$, satisfies $\dim T \leq \rank G$. Altogether,
$$\dim (\mathfrak{g}^*)^\Gamma \leq 2g+m+\rank G,$$
as claimed. This completes the proof of Lemma \ref{l:dualg}.
\end{proof}

\begin{proof}[Proof of Lemma \ref{l:inoutaut}.]
Without loss of generality, we can assume $ G$ is connected.
Let $\mathfrak{g}$ be the Lie algebra of $ G$. Then $\alpha$ acts also on  $\mathfrak{g}$, and
$\dim {\rm Fix}_{ G}(\alpha) = \dim \g^\alpha$, so we can work at the level of Lie algebras.
As $\alpha$ respects the decomposition of $\mathfrak{g}$ into $[\mathfrak{g},\mathfrak{g}]\oplus \mathfrak{z}$ where $\mathfrak{z}$ is the Lie algebra of the central torus. As $\rank \g = \rank [\g,  \g]+ \dim \mathfrak{z}$, we can restrict $\alpha$ to $[\g,\g]$ and assume $\g$ is semisimple.

Moreover we can write $\g$ as a direct sum $ \g=\bigoplus_{i=1}^s \g_i$ where each $\g_i$ is itself a direct sum of isomorphic simple Lie algebras such that for each $i$, $\alpha$ acts transitively on the simple components. As both sides of the inequality are additive on a direct sum of $\alpha$-invariant subalgebras, we can assume $\g$ is a sum of $t$ isomorphic simple algebras, $t|k$, and $\alpha$ acts transitively on the summands.
If $\alpha$ is inner, then $t=1$.  If $\alpha$ is pure outer, it is equivalent to an action of the form
$$\alpha(x_1,\ldots,x_t) = (\beta(x_t),x_1,\ldots,x_{t-1}),$$
where $\beta$ is a pure outer automorphism of a simple factor $\h$, of order $k/t$.  Thus,
$$\dim \g^\alpha = \dim \{(x,x,\ldots,x)\mid x\in \mathfrak{h}^\beta\} = \dim \mathfrak{h}^\beta.$$
Thus, for the outer case, it suffices to prove the result when $t=1$.  If $k=1$, the result is trivial.
The possibilities for $(\g,\h)$ are well-known (see, e.g., \cite[Chapter X, Table 1]{Helgason}).
For $k=2$, they are  $(\sl(2n),\sp(2n))$, $(\sl(2n+1),\so(2n+1))$,
$(\so(2n),\so(2n-1))$, and $(\mathfrak{e}_6,\mathfrak{f}_4)$, and for $k=3$, there is the unique case $(\so(8),\g_2)$.

Now assume $\alpha$ is inner. Here, (\ref{e:error}) follows from work of R. Lawther \cite{Law}.
We thank the referee for suggesting this reference.  For type A, a stronger estimate than (\ref{e:error}) holds, namely
$$\dim {\rm Fix}_{ G}(\alpha) \geq \frac{\dim  G}k- 1.$$
This will be needed for the upper bound in
Theorem~\ref{rep-thm} and is easy to see.  Namely, for
$x\in G = \SL_n$ of order $k$, let $a_j$ denote the multiplicity of $e^{2\pi i k/j}$ as an eigenvalue of $x$.
By the Cauchy-Schwartz inequality,
\begin{equation}
\label{A-upper}
\dim Z_G(x) + 1 = \sum_{j=0}^{k-1} a_j^2 \ge \frac{\Bigl(\sum_{j=0}^{k-1} a_j\Bigr)^2}k =
n^2 / k > \frac{\dim G}k.
\end{equation}
\end{proof}

\begin{proof}[Proof of Lemma \ref{l:inoutautgen}.]
To prove the statement, we still need to handle the case where $\alpha$ is neither an inner nor a pure outer automorphism. This means that for some $l$ dividing $k$, with $1<l<k$, $\alpha^l$ is inner while $\alpha$ is not. Let $ H=Z_{ G }(\alpha^l)={\rm Fix}_{ G}(\alpha^l)$. As $\alpha^l$ is an inner automorphism of order $k/l$, Lemma \ref{l:inoutaut} implies that
$$ \dim  H \geq \frac{\dim  G}{k/l}-\rank G.$$ Now $\alpha$ acts on the reductive group $ H$ as a pure outer automorphism of order at most $l$. Thus, again by Lemma \ref{l:inoutaut}
\begin{eqnarray*}
\dim {\rm Fix}_{ G}(\alpha) & = & \dim {\rm Fix}_{ H}(\alpha)\\
& \geq & \frac{\dim  H}{l}-\rank H\\
& \geq & \frac 1l\left( \frac{\dim G}{k/l}-\rank G \right)-\rank G\\
& \geq & \frac{\dim  G}{k}-\left(1+\frac1l \right)\rank G.
\end{eqnarray*}
As $l>1$, we get
$$\dim {\rm Fix}_{ G}(\alpha) \geq \frac{\dim  G}{k}-\frac{3}2 \rank G$$
completing the proof of Lemma \ref{l:inoutautgen}.\\
\end{proof}

In summary, we have proved the upper bounds for Theorems~\ref{rep-thm}, \ref{sl2-thm}, and \ref{so3-thm}.  For Theorems
\ref{sl2-thm} and \ref{so3-thm}, the bounds follow immediately from Proposition~\ref{p:uph1}, while the bound for Theorem~\ref{rep-thm} requires the better estimate
proved in (\ref{A-upper}).

\section{A Density Criterion}\label{s:three}

The results in this section are valid for general finitely generated groups $\Gamma$.
The main result  is Theorem~\ref{Density}, which gives a criterion for an irreducible component $C$ of $X_{\Gamma,G}$
to be contained in $X^{\epi}_{\Gamma,G}$, i.e. to have the property that there exists a Zariski-dense subset of $C(\bbR)$ consisting
of representations $\rho$ such that $\rho(\Gamma)$ is Zariski-dense in $G$.  We begin with the technical results needed in the proof
of Theorem~\ref{Density}.

\begin{prop}
\label{openness}
Let $G$ be a linear algebraic group over $\bbR$, and $H\subset G$ a closed subgroup such that $G(\bbR)/H(\bbR)$ is compact.
Let $C$ denote an irreducible component of $X_{\Gamma,H}$.
The condition on $\rho\in X_{\Gamma,G}(\bbR)$ that $\rho$ is not contained in any $G(\bbR)$-conjugate of
$C(\bbR)$ is open in the real topology.
\end{prop}

\begin{proof}

The conjugation map $H\times X_{\Gamma,H}\to X_{\Gamma,H}$ restricts to a map
$$H^\circ\times C\to X_{\Gamma,H}.$$
As $H^\circ$ and $C$ are irreducible, the image of this morphism lies in an irreducible component of $X_{\Gamma,H}$,
which must therefore be $C$.

The proposition can be restated as follows: the condition on $\rho$ that $\rho$ \emph{is} contained in some $G(\bbR)$-conjugate of $C(\bbR)$ is
closed in the real topology.
To prove this, consider a sequence $\rho_i\in X_{\Gamma,G}(\bbR)$ converging to
$\rho$.  Suppose that for each $\rho_i$ there exists $g_i\in G(\bbR)$ such that $\rho_i\in g_iC(\bbR)g_i^{-1}$.
Let $\bar g_i$ denote the image of $g_i$ in $G(\bbR)/H^\circ(\bbR)$.  As this set is compact, there exists a subsequence
which converges to some $\bar g\in G(\bbR)/H^\circ(\bbR)$.  Passing to this subsequence, we may assume that $\bar g_1,\bar g_2,\ldots$
converges to $\bar g$.  If $g\in G(\bbR)$ represents the coset $\bar g$, we claim that $\rho\in gC(\bbR)g^{-1}$.
The claim implies the proposition

By the
implicit function theorem, there exists a continuous section  $s\colon G(\bbR)/H^\circ(\bbR)\to G(\bbR)$ in a neighborhood of $\bar g$,
and we may normalize so that $s(\bar g) = g$.  For $i$ sufficiently large, $s(\bar g_i)$ is defined, and
$g_i = s(\bar g_i) h_i$ for some $h_i \in H^\circ(\bbR)$.
As conjugation by elements of $H^\circ(\bbR)$ preserves $C$,
we may assume without loss of generality that $g_i = s(\bar g_i)$ for all $i$ sufficiently large.
As $\lim_{i\to \infty} g_i = g$ and $C(\bbR)$ is closed in the real topology in $X_{\Gamma,G}(\bbR)$,
$$g^{-1} \rho g = \lim_{i\to \infty} g_i^{-1} \rho_i g_i \in C(\bbR).$$
\end{proof}

The following proposition is surely well-known, but for lack of a precise reference, we give a proof.

\begin{prop}
\label{maximals}
Let $G$ be an almost simple real algebraic group.  There exists a finite set $\{H_1,\ldots,H_k\}$ of proper closed subgroups of $G$ such that every proper closed subgroup
is contained in some group of the form $g H_i g^{-1}$, where $g\in G(\bbR)$.
\end{prop}

\begin{proof}
The theorem is proved for $G(\bbR)$ compact in \cite[1.3]{Larsen}, so we may assume henceforth that $G$ is not compact.

First we prove that every proper closed subgroup $K$ is contained in a maximal closed subgroup of positive dimension.  If $\dim K > 0$, then for every infinite ascending chain
$K_1 = K\subsetneq K_2 \subsetneq \cdots\subset G$ of closed subgroups of dimension $\dim K$, there exists a proper subgroup $L$ of $G$ which contains every $K_i$ and for which $\dim L > \dim K$.
Indeed, we can take $L := N_G(K^\circ)$, which contains all $K_i$, since $K_i^\circ = K^\circ$.  It cannot equal $G$ since $G$ is almost simple, and if $\dim K = \dim L$,
then $L^\circ = K^\circ$, and there are only finitely many groups between $K$ and $L$.  Thus every proper subgroup of $G$ of positive dimension is either contained in a maximal subgroup
of $G$ of the same dimension or in a proper subgroup of higher dimension.  It follows that each such proper subgroup is contained in a maximal subgroup.  For finite subgroups $K$,
we can embed $K$ in a maximal compact subgroup of $G$, which lies in a conjugacy class of proper closed subgroups of positive dimension since $G$ itself is not compact.

We claim that every  maximal closed subgroup $H$ of positive dimension is either parabolic or the normalizer of a connected semisimple subgroup or  the normalizer of a maximal torus.
Indeed, $H$ is contained in the normalizer of its unipotent radical $U$.  If $U$ is non-trivial, this normalizer is contained in a parabolic  $P$ \cite[30.3,~Cor. A]{Hump}, so $H=P$.
If $U$ is trivial, $H$ is reductive and is contained in the normalizer of the derived group of its identity component $H^\circ$.  If this is non-trivial, $H$ is the normalizer of a semisimple subgroup.
If not, $H^\circ$ is a torus $T$.  Then $H$ is contained in the normalizer of the derived group of $Z_G(T)^\circ$, which is again the normalizer of a semisimple subgroup unless $Z_G(T)^\circ$
is a torus.  In this case, it is a maximal torus, and $H$ is the normalizer of this torus.
Since a real semisimple group has finitely many conjugacy classes of parabolics and maximal tori, we need only consider the normalizers of semisimple subgroups.
There are finitely many conjugacy classes of these by a theorem of Richardson \cite{Richardson}.
\end{proof}

The proof of  Proposition~\ref{maximals} gives some additional information, which we employ in the following lemma:

\begin{lem}
\label{Non-parabolic}
If $H$ is a maximal proper subgroup of a split almost simple algebraic group $G$ over $\bbR$, then at least one of the following  statements is true:
\begin{enumerate}
\item $\dim H \le \frac 9{10}\dim G$.
\item  $H$ is a parabolic subgroup of $G$.
\item There exists an irreducible representation $V$ of $G$ of dimension $\le 2\rank G+1$ and a subspace $W$ of $V$ with $2\dim W < \dim V$
such that $\Stab_G W=H$.
\end{enumerate}
\end{lem}

\begin{proof}
For exceptional groups, all proper subgroups have dimension $\le \frac 9{10}\dim G$.  Indeed, if $G$ is an exceptional group over a finite fields $\bbF_q$ and $H$ is a closed subgroup over $\bbF_q$, then
the action of $G(\bbF_{q^n})$ on the set of $H(\bbF_{q^n})$-cosets gives a non-trivial complex representation of
degree $G(\bbF_{q^n})/H(\bbF_{q^n})$.  As $|H(\bbF_{q^n})| = O(q^{n\dim H})$,
the Landazuri-Seitz estimates for the minimal degree of a non-trivial complex representation of $G(\bbF_q)$ \cite{LZ} now imply $\dim H \le \frac 9{10}\dim G$.
The same result
follows in characteristic zero by a specialization argument.

We therefore consider only the case that $G$ is of type A, B, C, or D.  Also, we can ignore isogenies and assume that $G$ is either $\SL_n$,
a split orthogonal group, or a split symplectic group.  Let $V$ be the natural representation of $G$.
If $\dim V = n$, then $\dim G$ is $n^2-1$, $n(n-1)/2$, or $n(n+1)/2$, depending on whether $G$ is linear, orthogonal, or symplectic.
For linear groups, $G$ acts transitively on the set of subspaces of $V$ of given dimension, while for orthogonal and symplectic groups $G$, any $W_1$ and $W_2$ in $V$ for which the restrictions of the defining forms are isomorphic lie in the same orbit (see, e.g., Propositions 2 and 6 of \cite{Dieudonne}).

By the proof of Proposition~\ref{maximals}, we know that $H$ is the normalizer of a connected unipotent group, a maximal torus, or a semisimple subgroup of $G$.
If it is the normalizer of a non-trivial unipotent group or a maximal torus, we have (2) or (1) respectively.
We therefore assume that $H$ is the normalizer of a semisimple subgroup $K\subset G$, and it follows that $H^\circ$ preserves each irreducible factor of
$V$ as $K$-representation. 

If $W$ is an $H^\circ$-subrepresentation of dimension $m\le n/2$, then the $G$-orbit of $W$ in the Grassmannian of $m$-planes in $V$
has dimension $m(n-m)$ for $G=\SL_n$ and dimension at least 
$$m(n-m) - \binom{m+1}2$$
in the orthogonal and symplectic cases.  Indeed, if $W$ has basis $e_1,\ldots,e_m$, the restriction of the defining form $\langle\,,\,\rangle$ of $G$ to $W$ is 
determined by the values $\langle e_i,e_j\rangle$ for $1\le i\le j\le m$.  For $j=1,\ldots,m$ one can iteratively solve the system of equations 
$$\langle v_i,v_j\rangle = \langle e_i,e_j\rangle\ \forall i\le j$$
 to obtain a subvariety $X$ of $V^m$ of codimension $\le \binom{m+1}2$, and the open  subvariety of $X$ consisting of linearly independent $m$-tuples maps onto the $G$-orbit of $W$ with fiber dimension $\le m^2$,
thanks to the transitivity property of the $G$-action.
Thus, if neither (1) nor (3) is true, $V$ is $H^\circ$-irreducible.

We have therefore reduced to the case that  $K$ is semisimple and $V\otimes \bbC$ is irreducible, so we may and do extend scalars to $\bbC$
for the remainder of the proof.
If $K$ is not almost simple, then any element of $G$ which normalizes $K$ must respect a non-trivial tensor decomposition and therefore $H$ respects such a decomposition.
This implies
$$\dim H \le m^2+(n/m)^2-1 \le 3+n^2/4.$$
We may therefore assume that $K$ is almost simple and $V$ is associated to a dominant weight of $K$.  It is easy to deduce from the
Weyl dimension formula that every non-trivial irreducible representation
of a simple Lie algebra $L$ of rank $r$, other than the natural representation and its dual, has dimension at least $(r^2+r)/2$,
we need only consider the case that $V$ is a natural representation.  As $H\subsetneq G$, we need only consider the inclusions
$\SO(n) \subset \SL_n$ and $\Sp(n)\subset \SL_n$.  In all cases, we have $\dim H \le \frac 23\dim G$.

\end{proof}

We recall that $X_{\Gamma,G}^{\epi}$ is the  Zariski-closure in $X_{\Gamma,G}$ of
the set of Zariski-dense homomorphisms $\Gamma\to G(\bbR)$.  Given $\rho_0\colon \Gamma\to G$, if $H\subset G$
is a subgroup such that $\rho_0(\Gamma)\subset H(\bbR)$, we write $t_H := \dim Z^1(\Gamma,\mathfrak{h})$.

\begin{thm}
\label{Density}
Let $\Gamma$ be a finitely generated group, $G$ an almost simple real algebraic group,
and $\rho_0\in \Hom(\Gamma,G(\bbR))$ a non-singular $\bbR$-point of $X_{\Gamma,G}$.  Suppose that
for every maximal proper closed subgroup $H$ of $G$ at least one of the following is true:
\begin{enumerate}
\item $t_G-\dim G > \dim X_{\Gamma,H}-\dim H$.
\item $G(\bbR)/H(\bbR)$ is compact, and $\rho_0(\Gamma)$ cannot be conjugated into $H(\bbR)$.
\item $G(\bbR)/H(\bbR)$ is compact, $\rho_0(\Gamma)$ is a subgroup of $H'(\bbR)$ where $H'$ is conjugate to $H$, and
$$t_G-\dim G > t_{H'}-\dim H'.$$
\item There exists a representation $V$ of $G$ and a subspace $W\subset V$ such that $H$ stabilizes $W$ but $\rho_0(\Gamma)$ stabilizes no subspace of $V$
of dimension $\dim W$.
\end{enumerate}
Then $X^{\epi}_{\Gamma,G}$ contains the irreducible component of $X_{\Gamma,G}$ to which $\rho_0$ belongs.
\end{thm}

\begin{proof}
Let $C$ denote the irreducible component of $X_{\Gamma,G}$ containing $\rho_0$, which is unique since $\rho_0$ is a non-singular
point of $X_{\Gamma,G}$.  Again, since $\rho_0$ is a non-singular point, there is an open neighborhood $U$ of $\rho_0$ in $C(\bbR)$ which
is diffeomorphic to $\bbR^n$, where $n:= \dim C = \dim t_G$.

Let $\{H_1,\ldots,H_k\}$ represent the conjugacy classes of maximal
proper closed subgroups of $G$ given by Lemma~\ref{maximals}.
If, for some $i$, we can show that the closure of the set of homomorphisms with image in a conjugate of $H_i$
meets $C$ in a proper closed subset of $C$, then we can ignore all such morphisms.  More generally, if the set of homomorphisms with image in a conjugate
of $H_i$ meets $U$ in a closed set with empty interior, we can ignore $H_i$ since the complement of any subset of $U$ with empty interior 
remains dense in $U$ in the real topology and therefore dense in $C$ in the Zariski topology.  If we can ignore all $H_i$, the theorem holds.

Let $C_{i,j}$ denote the irreducible components of $X_{\Gamma,H_i}$.
For each component we consider the  conjugation morphism $\chi_{i,j}\colon G\times C_{i,j}\to X_{\Gamma,G}$.
We claim that the fibers of this morphism have dimension at least $\dim H_i$.  Indeed, the action of $H_i^\circ$
on $G\times C_{i,j}$ given by
$$h.(g,\rho_0) = (gh^{-1},h\rho_0 h^{-1})$$
is free, and $\chi_{i,j}$ is constant on the orbits of the action.  Thus, the image
of $\chi_{i,j}$ has dimension at most $\dim C_{i,j}+\dim G-\dim H_i$.  
If $H_i$ satisfies condition (1), then the image of $\chi_{i,j}$ has dimension less than $n$ for all $j$, so we can ignore 
representations whose images lie in a conjugate of $H_i$.

Suppose that $G(\bbR)/H(\bbR)$ is compact.  If $\rho_0$ does not belong to the image of $\chi_{i,j}$, 
then by Proposition~\ref{openness}, the image closure meets $U$ in a closed subset of $U$ without interior points, so we can ignore representations
$\Gamma\to G(\bbR)$ which can be conjugated into an element of $C_{i,j}$.  In particular, if $H_i$ satisfies condition (2), we can ignore all representations
which lie in a conjugate of $H_i$.
If $\rho_0$ does belong to the image of $\chi_{i,j}$
then without loss of generality we may assume $\rho_0(\Gamma)\subset H_i(\bbR)$,
and the dimension of the Zariski tangent space of $X_{\Gamma, H_i}$ at $\rho_0$ is greater than or equal to $\dim C_{i,j}$. 

If $H_i$ satisfies condition (3),
then again the dimension of  the image of $\chi_{i,j}$ is less than $n$, so again we can ignore representations which lie in a conjugate of $H_i$.

If $H_i$ satisfies condition (4), we use the fact that the subset of $X_{\Gamma,G}$ stabilizing a subspace of $V$ of dimension $\dim W$
is Zariski-closed to show that we can ignore all homomorphisms whose image lies in a conjugate of $H_i$.
It follows that  $X^{\epi}_{\Gamma,G}$ contains $C$.
\end{proof}

If $G$ is compact, $G(\bbR)/H(\bbR)$ is compact, so it is convenient to use only conditions (2) and (3).

\begin{cor}
\label{EasyDensity}
If $G$ is a compact  almost simple algebraic group over $\bbR$, $H$ is a connected maximal proper closed subgroup of $G$ with finite center,
and $\rho_0\colon \Gamma\to H(\bbR)$ has dense image, then $t_G - \dim G > t_H - \dim H$ implies
$X^{\epi}_{\Gamma,G}$ contains the irreducible component of $X_{\Gamma,G}$ to which $\rho_0$ belongs.
\end{cor}

\begin{proof}
To apply the theorem, we need only prove that $\rho_0$ is a non-singular point of $X_{\Gamma,G}$, since it is clear that condition (2)
of Theorem~\ref{Density} holds when $H_i$ is not conjugate to $H$ and condition (3) holds when $H_i$ is conjugate to $H$.
As $H$ is maximal, the product $Z_G(H) H$ must equal $H$, which means $Z_G(H) = Z(H)$ is finite.
Thus, $\mathfrak{g}^{\Gamma} = \mathfrak{g}^H = \{0\}$, and since $\mathfrak{g}$ is a self-dual $G(\bbR)$-representation,
this implies $(\mathfrak{g}^*)^{\Gamma} = \{0\}$, which implies that $\rho_0$ is a non-singular point of $X_{\Gamma,G}$.
\end{proof}

\section{The Alternating Group Method}

In this section $\Gamma$ is any (cocompact, oriented) Fuchsian group.
We first consider $G=\SO(n)$.

\begin{prop}\label{p:so}
For $\Gamma$ a Fuchsian group and $G=\SO(n)$, we have
$$\dim X^{\epi}_{\Gamma, \SO(n)} = (1-\chi(\Gamma))\dim \SO(n)+O(n)$$
where the implicit constant depends only on $\Gamma$.
\end{prop}

\begin{proof}
Proposition~\ref{p:uph1} gives the upper bound, so it suffices to prove
$$\dim X^{\epi}_{\Gamma, \SO(n)} \ge (1-\chi(\Gamma))\dim \SO(n)+O(n).$$

Let $d_1,\dots,d_m$ be defined as in (\ref{presentation}).  For large $n$, denote $C_i$, for $i=1,\dots,m$, the conjugacy class in the alternating group $\AG_{n+1}$ which consists of even permutations of $\{1,2,\ldots,n+1\}$ with only $d_i$-cycles and $1$-cycles and with as many $d_i$-cycles as possible.
Thus, any element of $C_i$ has at most $2d_i-1$ fixed points.
Theorem 1.9 of \cite{LS1} ensures that for large enough $n$, there exist epimorphisms $\rho_0$ from $\Gamma$ onto $\AG_{n+1}$, sending $x_i$ to an element of $C_i$ for $i=1,\dots,m$ and $x_i$ as in (\ref{presentation}).

Now $\AG_{n+1} \subset\SO(n)$ and moreover the action of $\AG_{n+1}$ on the Lie algebra $\so(n)$ of $\SO(n)$ is the restriction to $\AG_{n+1}$
of the irreducible $S_{n+1}$ representation associated to the partition $(n-1)+1+1$
(\cite[Ex.~4.6]{FH}).  If $n\ge 5$, this partition is not self-conjugate, so the restriction to $\AG_{n+1}$ is irreducible.
By (\ref{dim-z1}),
\begin{multline*}
\dim Z^1(\Gamma,{\rm Ad \circ \rho_0})=(1-\chi(\Gamma)) \dim \so(n) \\
					+\sum_{i=1}^m \left(\frac{\dim \so(n)}{d_i} -\dim \so(n)^{\langle x_i \rangle}\right).
\end{multline*}

Now $\dim \so(n)^{\langle x_i\rangle}$ is equal to the multiplicity of the eigenvalue 1 of $x=\rho_0(x_i)$ acting via ${\rm Ad}$ on $\so(n)$.
Note that the multiplicity of every
$d_i$th root of unity as an eigenvalue
for our element $x=\rho_0(x_i)$, when acting on the natural $n$-dimensional representation,  is of the form $\frac{n}{d_i}+O(1)$, where the implied constant depends only on $d_i$.  Identifying $\so(n)$ with the exterior square of the natural representation, we see that
$$\left | \frac{\dim \so(n)}{d_i} -\dim \so(n)^{\langle x_i \rangle}\right | = O(n),$$
where again the constant depends only on $d_i$.

As $\so(n)^*$ has no $\AG_{n+1}$-invariants, $X_{\Gamma,\SO(n)}$ is non-singular at $\rho_0$.
By Theorem~\ref{Density}, as long $n$ is large enough that
\begin{equation*}
\begin{split}
t_{\SO(n)} &= \dim Z^1(\Gamma,{\rm Ad \circ \rho_0}) \\
		&> \dim \SO(n) - \dim \AG_{n+1} + t_{\AG_{n+1}} \\
		&= \dim \SO(n),
\end{split}
\end{equation*}
$X^{\epi}_{\Gamma,\SO(n)}$ contains the component of $X_{\Gamma,\SO(n)}$ to which $\rho_0$ belongs, and this has dimension $t_{\SO(n)} =  (1-\chi(\Gamma))\dim \SO(n)+O(n)$.
\end{proof}

We remark that in this case, there is a more elementary alternative argument.
The condition on $X_{\Gamma,\SO(n)}$ of irreducibility on $\so(n)$ is open. It is impossible that all representations in a neighborhood of $\rho_0$ have finite image
and those with infinite image should have Zariski dense image (since the Lie algebra of the connected component of the Zariski closure is $\rho(\Gamma)$-invariant).

We can now prove Theorem~\ref{rep-thm}.

\begin{proof}
The upper bound has already been proved in \S1.  It therefore suffices to prove
$$\dim X^{\epi}_{\Gamma, \SU(n)} \ge (1-\chi(\Gamma))\dim \SU(n)+O(1).$$
Throughout the argument, we may always assume that $n$ is sufficiently large,

We begin by defining $\rho_0$ as in the proof of Proposition~\ref{p:so}.  Let $C$ denote the irreducible
component of $X_{\Gamma,\SO(n)}$ to which $\rho_0$ belongs.
We may choose
$\rho'_0\in C(\bbR)$ such that $\rho'_0(\Gamma)$ is Zariski-dense in $\SO(n)$.
As there are finitely many conjugacy classes of order $d_i$ in $\SO(n)$, the conjugacy class of $\rho(x_i)$ does not vary as $\rho$ ranges over the irreducible
variety $C$, so  $\rho_0(x_i)$ is conjugate to $\rho'_0(x_i)$ in $\SO(n)$.

We have no further use for $\rho_0$ and now redefine $\rho_0$ to be the composition of $\rho'_0$ with the inclusion $\SO(n)\hookrightarrow \SU(n)$.
The eigenvalues of $\rho_0(x_i)$ are $d_i$th roots of unity, and each appears with multiplicity $n/d_i + O(1)$, where the implicit constant may depend
on $d_i$ but does not depend on $n$.  The representation $\SO(n)\to \SU(n)$ is irreducible, so $(\su(n))^{\SO(n)} = \{0\}$.
As $\su(n)$ is a self-dual representation of $\SU(n)$, it is a self-dual representation of $\SO(n)$, so as $\rho_0(\Gamma)$ is dense in $\SO(n)$,
$$(\su(n)^*)^{\Gamma} = (\su(n)^*)^{\SO(n)} = \{0\}.$$
It follows that $X_{\Gamma,\SU(n)}$ is non-singular at $\rho_0$.  Since each eigenvalue of $\rho_0(x_i)$  has multiplicity $n/d_i + O(1)$,
$$t_{\SU(n)} =  \dim Z^1(\Gamma,{\rm Ad \circ \rho_0}) = (1-\chi(\Gamma))\dim \SU(n) + O(1).$$
We claim that $\SO(n)$ is contained in a unique maximal closed subgroup of $\SU(n)$.
Indeed, if $G$ is any intermediate group, the Lie algebra $\mathfrak{g}$ of $G$ must be an $\SO(n)$-subrepresentation of
$\su(n)$ which contains $\so(n)$.  Since $\su(n)/\so(n)$ is an irreducible $\SO(n)$-representation (namely, the symmetric square
of the natural representation of $\SO(n))$, it follows that $\mathfrak{g} = \su(n)$ or $\mathfrak{g} = \so(n)$.  In the former case,
$G=\SU(n)$.   In the latter case, $G$ is contained in $N_G(\SO(n))$.  This is therefore the unique maximal proper closed subgroup of $\SU(n)$
containing $\SO(n)$, or (equivalently) $\rho_0(\Gamma)$.  The theorem now follows from Theorem~\ref{Density}
together with the upper bound estimate Proposition~\ref{p:uph1}
applied to $N_G(\SO(n))$.
\end{proof}
We can also deduce Theorem~\ref{sl2-thm} for $G$ of type A and D from Proposition~\ref{p:so}.

\begin{proof}
If $G_1\to G_2$ is an isogeny, the morphism $X_{\Gamma,G_1}\to X_{\Gamma,G_2}$ is quasi-finite, and so
$$\dim X_{\Gamma,G_2}\ge \dim X_{\Gamma,G_1}.$$
Likewise, the composition of a homomorphism with dense image with an isogeny still has dense image, so
$$\dim X^{\epi}_{\Gamma,G_2}\ge \dim X^{\epi}_{\Gamma,G_1}.$$
In particular, to prove our dimension estimate for an adjoint group, it suffices to prove it for any covering group.
We begin by proving it for $G=\SL_n$, which also gives it for $\PGL_n$.

Let $\rho_0$ now denote a homomorphism $\Gamma\to \SO(n)\subset \SL_n(\bbR)$ with dense image and such that every
eigenvalue of $\rho_0(x_i)$  has multiplicity $n/d_i+O(1)$.  Such a homomorphism
exists by the proof of Proposition~\ref{p:so}.
It is well-known that $\SO(n)$ is a maximal closed subgroup of
$\SL_n$, and $\mathfrak{g}^{\SO(n)} = \{0\}.$   Thus $\rho_0$ is a non-singular point
of $X_{\Gamma,G}(\bbR)$.  Let $C$ denote the unique irreducible component to which it belongs.
In applying Theorem~\ref{Density}, we do not need to consider parabolic subgroups at all
since $\rho_0(\Gamma)$ is not contained in any and $G(\bbR)/H(\bbR)$ is compact when $H$ is parabolic.
All other maximal subgroups are reductive, and we may therefore apply Proposition~\ref{p:uph1}
to get an upper bound
$$\dim X_{\Gamma,H}\le (1-\chi(\Gamma))\dim H + 2g+m+(3m/2+1) n$$
By Lemma~\ref{Non-parabolic}, $\dim H < \frac 9{10}(n^2-1)$, so for $n$ sufficiently large,
$$\dim X_{\Gamma,H}-\dim H < \dim X_{\Gamma,G}-\dim G.$$
Thus condition (2) of Theorem~\ref{Density} holds, and so the component $C$ of $X_{\Gamma,G}$ to which $\rho_0$ belongs lies
in $X_{\Gamma,G}^{\epi}$.  It is therefore a non-singular point of $C$, and it follows that
$$\dim X_{\Gamma,G}^{\epi} \ge \dim C = \dim Z^1(\Gamma,\g) = (1-\chi(\Gamma))\dim \SL_n + O(n).$$

The argument for type D is very similar.  Here we work with $G=\SO(n,n)$, which is a double cover of
the split adjoint group of type $D_n$ over $\bbR$.  Our starting point is a homomorphism $\rho_0\colon \Gamma\to \SO(n)\times \SO(n)$
with dense image and such that the eigenvalues of
$$\rho(x_i)\in \SO(n)\times \SO(n)\subset \SO(n,n)\subset \GL_{2n}(\bbC)$$
have multiplicity $(2n)/d_i+O(1)$.  Such a $\rho_0$ is given by a pair $(\sigma,\tau)$ of dense homomorphisms $\Gamma\to \SO(n)$
satisfying a balanced  eigenvalue multiplicity condition and the additional condition that $\sigma$ and $\tau$ do not lie in the same orbit
under the action of $\Aut(\SO(n))$ on $X_{\Gamma,\SO(n)}$.  This additional condition causes no harm, since $\dim \Aut \SO(n) = \dim \SO(n)$,
while the components of $\dim X^{\epi}_{\Gamma,\SO(n)}$ constructed above (which satisfy the balanced eigenvalue condition) have dimension
greater than $\dim \SO(n)$ for large $n$.  Given a pair $(\sigma,\tau)$ as above, the closure $H$ of $\rho_0(\Gamma)$ is a subgroup of
$\SO(n)\times \SO(n)$ which maps onto each factor but which does not lie in the graph of an isomorphism between the two factors.
By Goursat's lemma, $H=\SO(n)\times \SO(n)$.  From here, one passes from $H$ to $G = \SO(n,n)$ just as in the case of groups of type A.

\end{proof}

\section{Principal Homomorphisms}\label{s:four}

It is a well-known theorem of de Siebenthal \cite{dS} and Dynkin \cite{Dy1} that
for every (adjoint) simple algebraic group $ G$ over $\bbC$
there exists a conjugacy class of \emph{principal} homomorphisms $\SL_2\to  G$
such that the image of any non-trivial unipotent element of $\SL_2(\bbC)$
is a regular unipotent element of $G(\bbC)$.
The restriction of the adjoint representation of $ G$ to $\SL_2$ via the principal homomorphism
is a direct sum of $V_{2e_i}$, where $e_1,\ldots,e_r$ is the sequence of exponents of $G$, and
$V_m$ denotes the $m$th symmetric power of the $2$-dimensional irreducible representation of
$\SL_2$, which is of dimension $m+1$ \cite{Kostant}.   In particular,
$$\dim G = \sum_{i=1}^r (2e_i+1),$$
where $r$ denotes $\rank G$.  As each $V_{2e_i}$ factors through $\PGL_2$, the same is true for the
homomorphism $\SL_2\to \Ad(G)$.
More generally, if $ G$ is defined and split over any field $K$ of characteristic zero, the principal homomorphism can be defined over $K$.

The following proposition is due to Dynkin:

\begin{prop}
\label{dynkin}
Let $ G$ be an adjoint simple algebraic group over $\bbC$ of type $A_1$, $A_2$, $B_n$ ($n\ge 4$), $C_n$ ($n\ge 2$), $E_7$, $E_8$, $F_4$, or $G_2$. Let $ H$ denote the image of a principal homomorphism of $ G$.  Let $ K$ be a  closed subgroup of $ G$ whose image in the adjoint representation of $ G$ is conjugate to that of $ H$.  Then $ K$ is a maximal subgroup of $ G$.

\end{prop}

\begin{proof}
As $ K$ is conjugate to $ H$ in $\GL(\g)$, in particular the number of irreducible factors of $\g$ restricted to $ H$ and to $ K$ are the same.  By \cite{Kostant}, this already implies that $ H$ and $ K$ are conjugate in $ G$.  The fact that $ H$ is maximal is due to Dynkin.  The classical and exceptional cases are treated in \cite{Dy3} and \cite{Dy2} respectively.
\end{proof}

As $\SL_2$ is simply connected, the principal homomorphism $\SL_2\to  G$ lifts to a
homomorphism $\SL_2\to  H$ if $ H$ is a split semisimple group which is simple modulo its center.
Again, this is true for split groups over any field of characteristic zero.
We also call such homomorphisms principal.

If $ G$ is an adjoint simple group over $\bbR$ with $ G(\bbR)$ compact
and $\phi\colon \PGL_{2,\bbC}\to  G_{\bbC}$ is a principal homomorphism over $\bbC$, $\phi$ maps
the maximal compact subgroup $\SO(3)\subset \PGL_2(\bbC)$ into a maximal compact subgroup
of $ G(\bbC)$.  Thus $\phi$ can be chosen to map $\SU(2)$ to $ G(\bbR)$, and such a homomorphism will again be called principal.
Likewise, if $ H$ is almost simple and $ H(\bbR)$ is compact, a principal homomorphism
$\phi\colon \SL_{2,\bbC}\to  H_{\bbC}$ can be chosen so that
$\phi(\SU(2))\subset  H(\bbR)$.

\begin{prop}
\label{Adjoint}
Let $ G$ be an adjoint compact simple real algebraic group of type $A_1$, $A_2$, $B_n$ ($n\ge 4$), $C_n$ ($n\ge 2$), $E_7$, $E_8$, $F_4$, or $G_2$,
and let $\Gamma$ be an $\SO(3)$-dense Fuchsian group.  Let $\rho_0\colon \Gamma\to G$ denote the composition of the map $\Gamma\to \SO(3)$ and the
principal homomorphism $\phi\colon \SO(3)\to G$.
If
\begin{multline*} -\chi(\Gamma)\dim G  + \sum_{j=1}^m \frac{\dim G}{d_j}- \sum_{j=1}^m\sum_{i=1}^r (1+ 2\lfloor e_i/d_j\rfloor)\\
	> -\chi(\Gamma)\dim \SO(3) +  \sum_{j=1}^m \frac{\dim \SO(3)}{d_j} -m,
\end{multline*}
then
\begin{equation}
\label{ex-ineq}
\dim X_{\Gamma,G}^{\epi} \ge(1-\chi(\Gamma))\dim G + \sum_{j=1}^m \frac{\dim G}{d_j} - \sum_{j=1}^m\sum_{i=1}^r (1+ 2\lfloor e_i/d_j\rfloor).
\end{equation}

\end{prop}

\begin{proof}

Let $x_j$ denote the $j$th generator of finite order in the presentation (\ref{presentation}).
If $\phi(x_j)$ lifts to an element of $\SU(2)$ whose eigenvalues are
$\zeta^{\pm 1}$, where $\zeta$ is a primitive $2d_j$-root of unity, the eigenvalues of the image
of $x_j$ in $\Aut(\g)$ are
$$\zeta^{-2e_1},\zeta^{2-2e_1},\zeta^{4-2e_1},\ldots,1,\ldots,\zeta^{2e_1},\zeta^{-2e_2},\ldots,\zeta^{2e_2},\ldots,\zeta^{-2e_r},\ldots,\zeta^{2e_r}.$$
The multiplicity of $1$ as eigenvalue is therefore $\sum_{i=1}^r (1+ 2\lfloor e_i/d_j\rfloor)$.
By (\ref{dim-z1}), the left hand side of (\ref{ex-ineq}) is $\dim Z^1(\Gamma,\g)$.
By Corollary~\ref{EasyDensity}, we need only check that
$$t_{G} - \dim G = -\chi(\Gamma)\dim G +  \sum_{j=1}^m \frac{\dim G}{d_j} - \sum_{j=1}^m\sum_{i=1}^r (1+ 2\lfloor e_i/d_j\rfloor).$$
is greater than
$$t_{\SO(3)} - \dim \SO(3) =  -\chi(\Gamma)\dim \SO(3) +  \sum_{j=1}^m \frac{\dim \SO(3)}{d_j} - \sum_{j=1}^m 1,$$
which is true by hypothesis.
\end{proof}

We can now prove Theorem~\ref{so3-thm}.

\begin{proof}
Recall that if $G_1\to G_2$ is an isogeny, we can prove the theorem for $G_1$ and immediately  deduce it for $G_2$.
Theorem~\ref{rep-thm} and Proposition~\ref{p:so} therefore cover groups of type A, B, and D.  This leaves only the symplectic case, where
Proposition~\ref{Adjoint} applies.
Note that
\begin{align*}
\label{difference}
\sum_{j=1}^m \frac{\dim G}{d_j} &-  \sum_{j=1}^m\sum_{i=1}^r (1+ 2\lfloor e_i/d_j\rfloor) \\
&= \sum_{j=1}^m\sum_{i=1}^r \frac{1+2e_i}{d_j} - \sum_{j=1}^m\sum_{i=1}^r (1+ 2\lfloor e_i/d_j\rfloor) \\
&= \sum_{i=1}^r \sum_{j=1}^m \Bigl(\frac{1+2e_i}{d_j} - 1+ 2\lfloor e_i/d_j\rfloor\Bigr).
\end{align*}
As
$$-1 < 2x+1/d_j - 1 - 2\lfloor x \rfloor < 1,$$
the error term is at most $mr$ in absolute value.
\end{proof}

The following proposition illustrates the fact that the methods of this section are not only useful in the large rank limit.
We  make essential use of the technique illustrated below in \cite{LLM}.

\begin{prop}
\label{Exceptional}
Every $\SO(3)$-dense Fuchsian group is also $F_4(\bbR)$-dense, $E_7(\bbR)$-dense, and $E_8(\bbR)$-dense,
where $F_4$, $E_7$, and $E_8$ denote the compact simple exceptional real algebraic groups of absolute rank $4$,
$7$, and $8$ respectively.
\end{prop}

\begin{proof}
Let $G$ be one of $F_4$, $E_7$, and $E_8$.  Let $E$ denote the set of exponents of $G$, other than $1$, which is the only exponent of $\SO(3)$.
We map  $\Gamma$ to $G(\bbR)$ via the principal homomorphism
$\SO(3)\to G$ and apply Corollary~\ref{EasyDensity}.  To show that there exists a homomorphism from $\Gamma$ to $G(\bbR)$
with dense image, we need only check that
$$t_G-\dim G > t_{\SO(3)}-\dim \SO(3).$$
The proof of Theorem~\ref{Density} proceeds by deforming the composed homomomorphism $\Gamma\to \SO(3)\to G(\bbR)$, and under continuous deformation,
the order of the image of a torsion element remains constant.
We therefore obtain more, namely that $\Gamma$ is $G(\bbR)$-dense.

By replacing $t_G$ and $t_{\SO(3)}$ by the middle expression in (\ref{e-dim}) for $V=\g$ and $V=\so(3)$ respectively, the desired
inequality can be rewritten
\begin{equation}
\label{positive}
(2g-2+m)(\dim G - \dim \SO(3)) -
\sum_{j=1}^m\sum_{e\in E} (1+2\lfloor e/d_j\rfloor) > 0.
\end{equation}
The summand
is non-increasing with each $d_j$.    In particular,
\begin{equation*}
\begin{split}
\sum_{j=1}^m\sum_{e\in E} (1+2\lfloor e/d_j\rfloor)
&\le \sum_{j=1}^m\sum_{e\in E} (1+2\lfloor e/2\rfloor)
< \sum_{j=1}^m\sum_{e\in E} (1+2e) \\
& = \dim G - \dim \SO(3).
\end{split}
\end{equation*}
Therefore, if $g\ge 1$, the expression (\ref{positive}) is positive.
For $g=0$, $(d_1,\ldots,d_m)$ is dominated by $(2,2,\ldots,2)$ for $m\ge 5$, $(2,2,2,3)$
for $m=4$, and $(2,3,7)$, $(2,4,5)$, or $(3,3,4)$ for $m=3$.

The following table presents the value of
$$\sum_{i=1}^r \Bigl((1+2\lfloor d_i/n\rfloor)-\frac{2d_i+1}{n}\Bigr)$$
for each root system of exceptional type and for each $n\le 7$.
\vskip 10pt
\begin{center}
\begin{tabular}{|c|r|r|r|r|r|r|r|r|}\hline
$n$&$A_1$&$E_6$&$E_7$&$E_8$&$F_4$&$G_2$ \\
\hline
$2$&$-1/2$&$-1$&$-7/2$&  $-4$&  $-2$&  $-1$ \\
$3$&$0$&$-2$&$-4/3$&$-8/3$&$-4/3$&$-2/3$ \\
$4$&$1/4$&$1/2$&$-1/4$&  $-2$&  $-1$& $1/2$ \\
$5$&$2/5$&$2/5$& $2/5$&$-8/5$& $8/5$& $6/5$ \\
$6$&$1/2$&$-1$&$-7/6$&$-4/3$&$-2/3$&$-1/3$ \\
$7$&$4/7$&$6/7$&$   0$& $4/7$& $4/7$&   $0$ \\
\hline
\end{tabular}
\end{center}
\vskip 10pt
By  (\ref{dim-z1}), the relevant values of $t_{G}-\dim G$ are given in the following table:
\vskip 10pt
\begin{center}
\begin{tabular}{|c|r|r|r|r|r|r|r|r|}\hline
$d_i$ vector&$A_1$&$E_6$&$E_7$&$E_8$&$F_4$&$G_2$ \\
\hline
$(2,2,2,3)$&$2$&$18$&$34$&$56$&$16$&$6$ \\
$(2,3,7)$&$0$&$4$&$8$&$12$&$4$&$2$\\
$(2,4,5)$&$0$&$4$&$10$&$20$&$4$&$0$\\
$(3,3,4)$&$0$&$10$&$14$&$28$&$8$&$2$\\
\hline
\end{tabular}
\end{center}
\vskip 10pt
For $(\underbrace{2,\ldots,2}_m)$, $m\ge 5$, the values of $t_G-\dim G$ for $A_1$, $E_6$, $E_7$, $E_8$, $F_4$, $G_2$  are
$2m-6$, $40m-136$, $70m-266$, $128m-496$, $28m-104$, $8m-28$ respectively.  In all cases except $(2,4,5)$ for $G_2$, the desired inequality holds.
\end{proof}

We conclude by proving Theorem~\ref{sl2-thm} in the remaining cases, i.e., for adjoint groups $G$ of type B or C.

\begin{proof}
We begin with a Zariski-dense homomorphism $\rho_0\colon \Gamma\to\PGL_2(\bbR)$.  Such a homomorphism always exists since $\Gamma$ is Fuchsian.
We now embed $\PGL_2$ via the principal homomorphism in a split adjoint group $G$ of type $B_n$ or $C_n$.  Assuming $n\ge 4$,
the image is a maximal subgroup, and we can apply Theorem~\ref{Density}
as in the A and D cases.
\end{proof}

\section{$\SO(3)$-dense Groups}\label{s:six}

In this section we show that almost all Fuchsian groups are $\SO(3)$-dense and classify the exceptions.

\begin{lemma}
\label{interval}
Let $d\ge 2$ be an integer.
\begin{enumerate}
\item If $d\neq 6$, there exists an integer $a$ relatively prime to
$d$ such that
$$\frac 14\le \frac ad\le \frac 12,$$
with equality only if $d\in\{2,4\}$.
\item If $d\not\in\{4,6,10\}$, then $a$ can be chosen such that
$$\frac 13\le \frac ad\le \frac 12,$$
with equality only if $d\in\{2,3\}$.
\item If $d\notin\{2,3,18\}$, there exists $a$ such that
$$\frac 1{12} < \frac ad < \frac 4{15},$$
with equality only if $d=12$.
\end{enumerate}

\end{lemma}

\begin{proof}
For (1) and (2), let
$$a=\begin{cases}
\frac{d-1}2&\text{if $d\equiv 1\pmod2$,}\\
\frac{d-4}2&\text{if $d\equiv 2\pmod4$,}\\
\frac{d-2}2&\text{if $d\equiv 0\pmod4$.}
\end{cases}$$
As long as $d > 12$, these fractions satisfy the desired inequalities, and for $d \le 12$, this can be checked by hand.

For (3), let $a=\frac{d-b}{6}$, where $b$ depends on $d$ (mod $36$) and is given as follows:
\vskip 10pt
\begin{tabular}{|r|c|c|}\hline
$b$&$d\pmod4$&$d\pmod9$ \\
\hline
$-12$&$2$&$3$\\
$-6$&$0$&$6$\\
$-4$&$2$&$2,5,8$\\
$-3$&$1,3$&$3$\\
$-2$&$0$&$1,4,7$\\
$-1$&$1,3$&$2,5,8$\\
$1$&$1,3$&$1,4,7$\\
$2$&$0$&$2,5,8$\\
$3$&$1,3$&$0,6$\\
$4$&$2$&$1,4,7$\\
$6$&$0$&$0,3$\\
$12$&$2$&$0,6$\\
\hline
\end{tabular}
\vskip 10pt
As long as $d > 24$, these fractions satisfy the desired inequalities, and the cases $d\le 24$ can be checked by hand.
\end{proof}

\begin{prop}
\label{base-case}
A cocompact oriented Fuchsian group is $\SO(3)$-dense if and only if it does not belong to the set
\begin{equation}
\label{bad-cases}
\{\Gamma_{2,4,6},\Gamma_{2,6,6},\Gamma_{3,4,4},\Gamma_{3,6,6}, \Gamma_{2,6,10},\Gamma_{4,6,12}\}.
\end{equation}
\end{prop}

\begin{proof}

We recall that every proper closed subgroup of $\SO(3)$ is contained in a subgroup of $\SO(3)$
isomorphic to $\O(2)$, $A_5$, or $S_4$.  The set of homomorphisms $\O(2)\to \SO(3)$,
$A_5\to \SO(3)$, and $S_4\to \SO(3)$ have dimension $2$, $3$, and $3$ respectively.
Furthermore, $\dim  X_{\Gamma,\O(2)}\le 2g+m$, while
$\dim  X_{\Gamma,S_4} = \dim  X_{\Gamma,A_5} = 0$.

Every non-trivial conjugacy class in $\SO(3)$ has dimension $2$.
As the commutator map $\SO(3)\times \SO(3)\to\SO(3)$ is surjective and every fiber has dimension
at least $3$, if $g\ge 1$, we have $\dim  X_{\Gamma,\SO(3)} \ge 3+3(2g-2)+2m$.  For $g\ge 2$
or $g=1$ and $m\ge 2$, the dimension of $\dim  X_{\Gamma,\SO(3)}$ exceeds the dimension of
the space of all homomorphisms whose image lies in a proper closed subgroup, so
there exists a homomorphism with dense image with $\rho(x_i)$ of order $d_i$ for all $i$.
If $g=m=1$, and $\rho(\Gamma)\subset \O(2)$,
then the commutator
$\rho([y_1,z_1])$ lies in $\SO(2)$, so $\rho(x_1)\in\SO(2)$.  The set of elements of order
$d_1$ in $\SO(2)$ is finite, so $\dim  X_{\Gamma,\O(2)} \le 2$, and the set of elements of
$ X_{\Gamma,\SO(3)}$ which can be conjugated into a fixed $\O(2)$ has dimension
$\le 4$; again there exists $\rho$ with dense image and with $\rho(x_i)$ of order $d_i$ for all $i$.

This leaves the case $g=0$, $m\ge 3$.  By (\ref{Euler}), $\sum 1/d_i < m-2$.  We claim that
unless we are in one of the cases of (\ref{bad-cases}),
there exist elements $\bar x_1,\ldots,\bar x_m\in\SO(3)$ of orders $d_1,\ldots,d_m$ respectively
such that $\bar x_1\cdots\bar x_m=e$ and the elements $\bar x_i$ generate a dense subgroup of $\SO(3)$.
For $m=3$, the order of terms in the sequence  $d_1,d_2,d_3$
does not matter since $\bar x_1\bar x_2\bar x_3 = e$ implies $\bar x_2\bar x_3\bar x_1=e$
and $\bar x_3^{-1}\bar x_2^{-1}\bar x_1^{-1} = e$.
Without loss of generality we may therefore assume that
$d_1\le d_2\le d_3$ when $m=3$.
If the base case $m=3$ holds whenever $d_3$ is sufficiently large, the higher $m$ cases follow by induction, since one can replace
the $m+1$-tuple $(d_1,\ldots,d_{m+1})$ by the $m$-tuple $(d_1,\ldots,d_{m-1},d)$
and the triple $(d_m,d_{m+1},d)$, where $d$ is sufficiently large.

If $\alpha_1, \alpha_2, \alpha_3\in (0,\pi]$ satisfy the triangle inequality,
by a standard continuity argument,
there exists a non-degenerate spherical triangle whose sides have angles $\alpha_i$.
If $\alpha_1$, $\alpha_2$, and $\alpha_3$ are of order $d_1$, $d_2$, and $d_3$ respectively, then there exists a homomorphism from the triangle group $\Gamma_{d_1,d_2,d_3}$ to $\SO(3)$
such that the generators $x_i$ map to elements of order $d_i$, and these elements do not commute.
We claim that except in the cases $(2,4,6)$, $(2,6,6)$, $(3,6,6)$,  $(2,6,10)$, and $(4,6,12)$, there always exist
positive integers $a_i\le d_i/2$ such that $a_i$ is relatively prime to $d_i$ and
$a_i/d_i$ satisfy the triangle inequality.  We can therefore set $\alpha_i =2a_i\pi/d_i.$

Every non-decreasing triple from the interval $[1/4,1/2]$ except for $1/4,1/4,1/2$ satisfies the triangle inequality.  As $(d_1,d_2,d_3)$ cannot be $(2,4,4)$, Lemma~\ref{interval} (1) implies the
claim unless at least one of $d_1,d_2,d_3$ equals $6$.
We therefore assume that at least one of the $d_i$ is $6$.
As $1/6$ and any two elements
of $[1/3,1/2]$ other than $1/3$ and $1/2$ satisfy the triangle inequality and as
$(d_1,d_2,d_3)\neq(2,3,6)$, Lemma~\ref{interval} (2) implies the claim except
if one of the $d_i$ is $4$, one
of the $d_i$ is $10$, or two of the $d_i$ are $6$.  By Lemma~\ref{interval} (3), the remaining $a_i/d_i$ can then be chosen to
lie in $(1/12,4/15)$ unless this $d_i\in \{2,3,12,18\}$.
If $a_i/d_i$ is in this interval, the triangle inequality follows.  Examination of the remaining $12$ cases reveal five exceptions:
$(2,4,6)$, $(2,6,6)$, $(2,6,10)$, $(3,6,6)$, and $(4,6,12)$.

Assuming that we are in none of these cases, there exist non-commuting elements $\bar x_i$ in
$\SO(3)$ of order $d_1$, $d_2$, and $d_3$,  such that $\bar x_1 \bar x_2 \bar x_3 = e$.
They cannot all lie in a common $\SO(2)$.  In fact, they cannot  all lie in a common $\O(2)$, since any element in
the non-trivial coset of $\O(2)$ has order $2$, $d_3\ge d_2 > 2$, and if three elements multiply to the identity, it is impossible that exactly two lie in $\SO(2)$.  If $\Gamma$ maps to $S_4$ or
$A_5$, then $\{d_1,d_2,d_3\}$ is contained in $\{2,3,4\}$ or $\{2,3,5\}$ respectively.
The possibilities for
$(d_1,d_2,d_3)$ are therefore $(2,5,5)$, $(3,3,5)$, $(3,5,5)$, $(5,5,5)$,
$(3,4,4)$, $(3,3,4)$, and $(4,4,4)$.   The realization of $\Gamma_{a,b,b}$ as an index-$2$ subgroup of
$\Gamma_{2,2a,b}$ implies the proposition for
$\Gamma_{2,5,5}$, $\Gamma_{3,3,5}$, $\Gamma_{3,5,5}$, $\Gamma_{5,5,5}$, $\Gamma_{3,3,4}$, and $\Gamma_{4,4,4}$.  The only remaining case is $\Gamma_{3,4,4}$.

Lastly, we show that none of the groups in (\ref{bad-cases}) are $\SO(3)$-dense.    Suppose there exist elements $x_1,x_2,x_3$ of orders $d_1,d_2,d_3$ respectively such that $x_1x_2x_3$ equals the identity and
$\langle x_1,x_2,x_3\rangle$ is dense in $\SO(3)$.  These elements can be regarded as rotations through
angles $2\pi a_1$, $2\pi a_2$, $2\pi a_3$ respectively, where the $a_i$ can be taken in $[0,1/2)$, and no two axes of rotation coincide.  Choosing a point $P$ on the great circle of vectors perpendicular to the axis of rotation of $x_1$, the three points
$P, x_2^{-1}(P), x_1(P) = x_3^{-1} x_2^{-1}(P)$ satisfy the strict spherical triangle inequality, so $a_1 < a_2+a_3$. Likewise $a_2 < a_3+a_1$ and $a_3 < a_1+a_2$.
However, one easily verifies in each of the cases (\ref{bad-cases}) that one cannot find rational numbers $a_1,a_2,a_3\in (0,1/2]$ with denominators $d_1$, $d_2$, $d_3$ respectively such that $a_1,a_2,a_3$ satisfy the strict triangle inequality.
\end{proof}

\section{Appendix by Y. William Yu}

The following triples of permutations, which evidently multiply to $1$, have been checked by machine to generate the full alternating groups in which they lie:

\begin{itemize}

\item $\Gamma_{2,4,6}\to \AG_{14}$:
\begin{equation*}
\begin{split}
x_1 &= (1\;2)(3\;4)(5\;6)(7\;8)(9\;10)(11\;12) \\
x_2 &= (1\;10\;9\;8)(2\;14\;13\;3)(4\;5)(6\;7\;12\;11) \\
x_3 &= (1\;3\;5\;11\;7\;9)(2\;8\;6\;4\;13\;14)
\end{split}
\end{equation*}

\item $\Gamma_{2,6,6}\to \AG_{14}$:
\begin{equation*}
\begin{split}
x_1 &= (1\;2)(3\;4)(5\;6)(7\;8)(9\;10)(11\;12) \\
x_2 &= (1\;14\;8\;7\;4\;2)(3\;5\;13\;11\;9\;6) \\
x_3 &= (1\;4\;6\;3\;7\;14)(5\;9\;10\;11\;12\;13)
\end{split}
\end{equation*}

\item $\Gamma_{3,6,6}\to \AG_{12}$:
\begin{equation*}
\begin{split}
x_1 &= (1\;2\;3)(4\;5\;6)(7\;8\;9)(10\;11\;12) \\
x_2 &= (1\;12\;11\;6\;2\;3)(4\;10\;8\;9\;5\;7) \\
x_3 &= (1\;2\;3\;6\;9\;10)(4\;11)(5\;7\;8)
\end{split}
\end{equation*}

\item $\Gamma_{3,4,4}\to \AG_{14}$:
\begin{equation*}
\begin{split}
x_1 &= (1\;2\;3)(4\;5\;6)(7\;8\;9)(10\;11\;12) \\
x_2 &= (1\;14\;11\;12)(2\;3\;4\;5)(7\;10\;13\;9)(6\;8) \\
x_3 &= (1\;2\;12\;14)(3\;5)(4\;8\;9\;6)(7\;13\;10\;11)
\end{split}
\end{equation*}

\item $\Gamma_{2,6,10}\to \AG_{12}$:
\begin{equation*}
\begin{split}
x_1 &= (1\;2)(3\;4)(5\;6)(7\;8)(9\;10)(11\;12) \\
x_2 &= (1\;8\;6\;7\;5\;3)(4\;10\;11)(9\;12) \\
x_3 &= (1\;2\;3\;11\;9\;4\;5\;8\;6\;7)(10\;12)
\end{split}
\end{equation*}

\item $\Gamma_{4,6,12}\to \AG_{12}$:
\begin{equation*}
\begin{split}
x_1 &= (1\;4\;3\;2)(5\;8\;7\;6)(9\;10)(11\;12) \\
x_2 &= (1\;2\;5\;9\;10\;3)(4\;7\;11\;8\;6\;12) \\
x_3 &= (2\;10\;5\;8)(3\;12\;7\;11\;6\;4)
\end{split}
\end{equation*}

\end{itemize}

In each case, one can use (\ref{e-dim}) to compute that
$$\dim Z^1(\Gamma,\so(n)) - \dim \SO(n)> 0.$$
The reasoning of Proposition~\ref{p:so}
therefore applies to give a homomorphism $\Gamma\to \SO(n)$ either for $n=11$ or for $n=13$, with dense image.

\end{document}